\documentclass{article}
\usepackage{amsmath}
\usepackage{amsthm}
\usepackage{mathtext}
\usepackage{amssymb}
\usepackage{tocenter}

\ToCenter{450pt}{650pt}

\newtheorem{thm}{Theorem}
\newtheorem*{thm*}{Theorem}
\newtheorem{cor}{Corollary}
\newtheorem{lem}{Lemma}
\newtheorem*{lem*}{Lemma}
\newtheorem{prop}{Proposition}
\newtheorem*{con*}{Conjecture}
\newtheorem*{prob*}{Problem}
\theoremstyle{definition}
\newtheorem{defn}{Definition}
\newtheorem*{ex*}{Example}

\newtheorem*{cons*}{Construction}

\theoremstyle{remark}
\newtheorem*{not*}{Notation}

\DeclareMathOperator{\lk}{lk}
\DeclareMathOperator{\st}{st}

\begin{document}
\title{Combinatorics of flag simplicial 3-polytopes.}%
\author{V.~D.~Volodin\thanks{This work is supported by the Russian Government project 11.G34.31.0053 and program 2 OM RAN.}}%
\date{}

\maketitle
\begin{abstract}
In the focus of this paper is the operation of edge contraction. One can show that simplicial 3-polytope is flag iff contraction of any its edge gives simplicial 3-polytope. 
Our main result states that any flag simplicial 3-polytope can be reduced to octahedron by sequence of edge contractions. Using this operation we introduce a partial order on the set of flag simplicial 3-polytopes and study Hasse graph of corresponding poset.  We estimate input and output degrees of vertices of this Hasse graph.
\end{abstract}

\section{Introduction and main definitions}
We consider simplicial polytopes of dimension 3. Boundary of such a polytope is a simplicial sphere of dimension 2. Inverse is also true, every 2-dimensional simplicial sphere is a boundary of some simplicial 3-polytope. So, we can consider combinatorics simplicial spheres and forget about their polytopal realization.

\begin{defn}
Simplicial complex is called \emph{flag}  if every its clique forms a simplex. Simplicial polytope is called \emph{flag} if its boundary is a flag simplicial complex.
\end{defn}

Simple polytopes are dual to simplicial ones. Simple polytope is called \emph{flag} if its dual simplicial polytope is flag. In terms of face lattice it can be formulated as following.

\begin{defn}
Simple polytope is called  \emph{flag} if every collection of its pairwise intersecting faces has a nonempty intersection.
\end{defn}

\begin{defn}
Contraction of edge $e=\{v_1,v_2\}$ in some simlicial complex $K$ is replacing the union of stars $\st_K v_1\cup \st_K v_2$ by the star $\st_K v$ of new vertex $v$. The obtained complex is denoted by $K/e$.
This operation also can be described as topological shrinking of the edge $e$, identifying multiple edges and removing degenerated faces. There exists a natural simplicial map $K\to K/e$ which merges vertiices $v_1$ and $v_2$ into the vertex $v$ and is bijective on the complements to these vertices.
\end{defn}

Operation of edge contraction on simplicial complexes was actively studied in \cite{DEGN, Me}. In \cite{DEGN} there was formulated condition that provides same topological type of resulting complex. The main result of the present paper is following.

\begin{thm*}
Let $P$ be a flag simplicial 3-polytope. Then, $P$ can be reduced to octahedron by sequence of edge contractions.
\end{thm*}

In \cite{Me} there was proved that simplicial complex of dimension not greater than 2 can be obtained from any its geometric subdivision by sequence of edge contractions. So the above theorem could be proved by showing that boundary of any flag simplicial 3-polytope is geometric subdivision of the boundary of octahedron. But we give a direct proof based on the notion of \emph{belt} in simplicial polytope. 

If the initial simplicial complex is flag simplicial sphere, then obtained simplicial complex may fail to be flag simplicial sphere. After topological edge contraction boundary of the polytope becomes cell complex homeomorphic to sphere, Therefore,  the resulting complex may fail to be simplicial sphere only after identifying multiple edges and faces.

\begin{prop}
Simplicial sphere $K^2$ is a flag iff for any its edge $e$ the complex $K^2/e$ is still simplicial sphere.
\end{prop}
\begin{proof}
If $K^2$ has missing face $\{v_1,v_2,v_3\}$, then contracting the edge $\{v_1,v_2\}$ into vertex $v$ we obtain a union of two spheres glued by common edge $\{v, v_3\}$.

If complex $K^2/e$ is not polytopal, then using Steinitz's theorem we can find an edge $e'$ containing the image of collapsed edge $e$, such that complex  $(K^2/e )\setminus e'$ is disconnected. Then, the preimage of the edge $e'$ is a missing face in complex $K^2$, which is flag by assumption.
\end{proof}

The inverse operation can be better described in dual terms. Edge contraction of simplicial polytope is dual to merging two intersecting facets of simple polytope by erasing edge between them. Operation inverse to erasing edge can be described as following. Let $P$ be a flag simple 3-polytope and $F$ be its facet. Lets pertube $F$ by two different ways and obtain two facets $F_1$ and $F_2$ intersecting in some edge $e$. Denote by $Q$ the obtained simple polytope. Notice, that polytope $P$ is obtained from polytope $Q$ by erasing the edge $e$. Let's call described operation by \emph{cutting facet} $F$ into facets $F_1$ and $F_2$.

\section{Partial order on flag simplicial 2-spheres}
Let's introduce partial order on the set of simplicial 2-spheres (or 3-polytopes). We assume that $P\preceq Q$ if $Q$ can be obtained from $P$ by sequence of edge contractions. Further we will study graph $\Gamma$ defined as Hasse diagram of the poset of flag simplicial 3-polytopes. Minimal polytopes with respect to this order are polytopes such that no edge can be contracted. Such polytopes correspond to initial vertices of $\Gamma$ and will be shortly called minimal.
 
\begin{thm}\label{edge_collapse_theorem}
Let $P^3$ be a flag simplicial 3-polytope. Then, $P^3$ can be reduced to octahedron by sequence of edge contractions.
\end{thm}
\begin{cor}
The vertex of graph $\Gamma$ corresponding to octahedron is the only initial vertex of graph $\Gamma$.
\end{cor}

\begin{defn}
We say that vertex set $\{v_1,v_2,v_3,v_4\}$ of simplicial polytope forms a \emph{belt} $\square$, if simplicial subcomplex generated by this set is boundary of square.
\end{defn}

\begin{lem}
Let $K^2$ be a simplicial 2-sphere and $e$ be some its edge. Simplicial sphere $K^2/e$ is flag iff the edge $e$ is not contained in any belt of $K^2$.
\end{lem}
\begin{cor}
Let $v_K$ be a vertex of $\Gamma$ corresponding to simplicial sphere $K^2$. Then, input degree of $v_K$ in graph $\Gamma$ is not greater than number of edges of complex $K^2$ not contained in any belt.
\end{cor}
\begin{proof} 
The idea of the proof is to show that belts in $K$ are preimages of missing faces of $K/e$ with respect to simplicial map $K\to K/e$.

If the edge $e=\{v_1,v_2\}$ is contained in some belt $\{v_1,v_2,v_3,v_4\}$ in $K$, then $K/e$ contains a missing face $\{v,v_2,v_3\}$.

Assume that after contaction of edge $e=\{v_1,v_2\}\subset K$ into the vertex $v$ obtained complex $K/e$ is not flag. Consider minimal missing face $V=\{v_i\}_I$ of complex $K/e$.  From the minimality of $V$ follows that $|V| = 3$, and since complex $K$ is flag, then $v\in V$. Let $V=\{v, v_3,v_4\}$, then $v_3,v_4\in \lk_{K}v_1 \cup \lk_{K}v_2$, but since  $V\notin K/e$, one of the vertices $v_3$ and $v_4$ is contained in $\lk_{K}v_1\setminus\lk_{K}v_2 $ and another vertex is contained in $\lk_{K}v_2\setminus\lk_{K}v_1$. Therefore, $\{v_1,v_2,v_3, v_4\}$ is a belt containing edge $e$.
\end{proof}

\begin{lem}\label{square}
Let $K^2$ be a minimal flag simplicial 2-sphere. Then, there exists a vertex $w\in K^2$, such that its link is a boundary of square.
\end{lem}
\begin{proof}
Each belt $\square$ divides sphere $K$ into simplicial balls $W^1$ and $W^2$.  Consider an arbitrary vertex $v\in K$. 

There exists a belt $\square_0$ containing  $v$, such that every belt $\square'$ containing $v$ and intersecting interior of $W_{0}^1$ intersects interior of $W_{0}^2$. Indeed, consider some belt $\square_1$ containing $v$. It divides simplicial sphere $K$ into simplicial balls $W_{1}^1$ and $W_{1}^2$.  If the belt $\square_2$ is contained in $W_{1}^1$, then consider $\square_2$. It divides $K$ into $W_{2}^1$ and $W_{2}^2$ and $W_{2}^1 \subset W_{1}^1$. Further we continue to choose $\square_i$ such that $W_{i}^1 \subset W_{i-1}^1$. On some step choosing new belt will be not possible and the last chosen belt $\square_0$ will satisfy the required property.  

Let $w\in \square_0$ and $\{v,w\}\notin K$. Each edge $\{w', v\}$ with $w'$ from interior $W_{0}^1$ is contained in some belt $\square_{w'}$ intersecting interior $W_{0}^2$. Then, $\square_0 \cap\square_{w'}=\{v, w\}$. Therefore, all vertices from $W_{0}^1\cap\lk_K v$ are adjacent to vertex $w$. Since $K$ is a flag simplicial 2-sphere,  $W_{0}^1$ doesn't have other vertices. Since $W_{0}^1$ satisfies required property, then $|W_{0}^1|=1$, otherwise there exists a belt containing internal vertex of $W_{0}^1$. Therefore $\lk_K w=\square_0$.
\end{proof}

\begin{proof}[Proof of Theorem \ref{edge_collapse_theorem}]
Let $K$ is minimal flag simplicial 2-sphere. Prove that $K$ is a boundary of octahedron.
Let $w$ be a vertex from Lemma \ref{square}. Consider a belt $\square_1$ containing $w$ and a belt $\square_2$ containing $w$ and vertices from $\lk_{\partial K}w\setminus \square_1$. Belts $\square_1$ and $\square_2$ have common vertex $w$, therefore they have one more intersection point $v$. Thus, all the vertices from $\lk_K w$ are adjacent to both $w$ and $v$. Since $K$ is flag simplicial 2-sphere, it doesn't have other vertices, and, therefore is a boundary of octahedron. 
\end{proof}

\begin{lem}
Let polytope $Q$ be obtained from flag simple polytope $P$ by cutting a face $F$ into faces $F_1$ and $F_2$. Polytope $Q$ is flag if and only if neither $F_1$ nor $F_2$ is triangle.
\end{lem}
\begin{proof}
Necessity is clear, since simple flag polytope doesn't have triangle faces. Assume, polytope $Q$ is not flag and $\mathcal{F} = \{F_i\}_I$ is a minimal collection of pairwise intersecting facets having empty intersection. Obviously, $\mathcal{F}\cap\{F_1,F_2\}\neq\emptyset$. Without loss of generality assume that $F_1\in \mathcal{F}$. If we replace set $F_1$ by $F_1\cup F_2$, then obtained set collection will have nonempty itersection, since $P$ is flag. Therefore, all the facets from $\mathcal{F}$ intersect $F_2$,  and $F_2$ is a simple 2-polytope which is not flag, then $F_2$ is a triangle. 
\end{proof}

\begin{cor}
Let $v_K$ be a vertex of $\Gamma$, corresponding flag simplicial sphere $K$. Let $r_k$ be a number of $k$-angle 2-faces of dual simple polytope $P_K$. Then, output degree of vertex $v_K$ is not greater, then $$\sum_k r_k \frac{k(k-3)}{2}.$$
\end{cor}

\begin{proof}
Consider simple polytope $P$, dual to simplicial sphere $K$. Let $Q$ be a minimal polytope that is greater, then $P$. Then polytope $Q$ is obtained from $P$ by cutting its face $F$ into faces $F_1$ and $F_2$, which are not triangles. If the face $F$ is a $k$-angle, the number of ways to cut it without triangles is equal to the number of diagonals of $k$-angle, which is 
$$\sum_{i=0}^{k-4}(i+2)=\frac{k(k-3)}{2}$$
\end{proof}

\small\bigskip
\textsc{Steklov Mathematical Institute,Moscow,Russia}\\
\textsc{Delone Laboratory of Discrete and
Computational Geometry,Yaroslavl State University,Yaroslavl,Russia}\\
\emph{E-mail adress:} \verb"volodinvadim@gmail.com"

\end{document}